\documentclass{ecai} 

\usepackage{latexsym}
\usepackage{amssymb}
\usepackage{amsmath}
\usepackage{amsthm}
\usepackage{booktabs}
\usepackage{enumitem}
\usepackage{graphicx, subfigure}
\usepackage{color}
\usepackage{eurosym}
\usepackage{float}
\usepackage[switch]{lineno}
\usepackage[ruled,vlined,linesnumbered]{algorithm2e}

\newtheorem{theorem}{Theorem}

\newtheorem{proposition}[theorem]{Proposition}

\newtheorem{definition}{Definition}

\newcommand{\KwInit}[1]{\textbf{Initialization:} #1\\}

\newcommand{\BibTeX}{B\kern-.05em{\sc i\kern-.025em b}\kern-.08em\TeX}

\begin{document}


\begin{frontmatter}




\title{A Penalty-Based Guardrail Algorithm for Non-Decreasing Optimization with Inequality Constraints}


\author[1]{\fnms{Ksenija}~\snm{Stepanovic}\orcid{0000-0003-4273-4785}\thanks{Corresponding Author. Email: K.Stepanovic@tudelft.nl}}
\author[1]{\fnms{Wendelin}~\snm{Böhmer}\orcid{0000-0002-4398-6792}}
\author[1]{\fnms{Mathijs}~\snm{de Weerdt}\orcid{0000-0002-0470-6241}}

\address[1]{Delft University of Technology, Delft, The Netherlands}


\begin{abstract}
Traditional mathematical programming solvers require long computational times to solve constrained minimization problems of complex and large-scale physical systems. Therefore, these problems are often transformed into unconstrained ones, and solved with computationally efficient optimization approaches based on first-order information, such as the gradient descent method.
However, for unconstrained problems, balancing the minimization of the objective function with the reduction of constraint violations is challenging. We consider the class of time-dependent minimization problems with increasing (possibly) nonlinear and non-convex objective function and non-decreasing (possibly) nonlinear and non-convex inequality constraints. To efficiently solve them, we propose a penalty-based guardrail algorithm (PGA). This algorithm adapts a standard penalty-based method by dynamically updating the right-hand side of the constraints with a guardrail variable which adds a margin to prevent violations. We evaluate PGA on two novel application domains: a simplified model of a district heating system and an optimization model derived from learned deep neural networks. Our method significantly outperforms mathematical programming solvers and the standard penalty-based method, and achieves better performance and faster convergence than a state-of-the-art algorithm (IPDD) within a specified time limit.
\end{abstract}

\end{frontmatter}


\section{Introduction}
Many real-world physical systems, such as energy systems, can be formulated as large-scale optimization problems across multiple time steps, with decision variables coupled in inequality constraints~\cite{cannon2003dynamics}. The solution or set of solutions derived from an optimization should be feasible, that is, satisfy all constraints. Therefore, these problems are often modeled using a mathematical programming (MP) framework and solved with an off-the-shelf MP solver. Well-known examples of these solvers include Gurobi~\cite{gurobi}, SCIP~\cite{10.1145/3585516}, and BARON~\cite{nohra2021sdp}. Besides producing feasible solutions, the computational time of an optimization algorithm should usually remain low to be able to react in time. Concerning linear or small-scale models, MP solvers can meet this requirement. However, if models are complex or large, maintaining low computational time becomes a challenge~\cite{boyd2004convex}. Consequently, the research community has been exploring computationally more efficient optimization algorithms based on first-order information, such as the gradient descent method~\cite{amari1993backpropagation,tseng2009coordinate}.

However, optimization using the gradient descent method suffers from two challenges. The first challenge concerns finding feasible solutions in which the value of the objective function is near-optimal. A common approach involves transforming the constrained optimization problem into an unconstrained one by integrating constraints in the objective function using penalty-based~\cite{hestenes1969multiplier}, barrier~\cite{ben1997penalty} or augmented Lagrangian (AL)~\cite{hestenes1969multiplier,bertsekas2014constrained} approaches. In these problems, it is challenging to balance the minimization of the objective function and reduction of constraint violations~\cite{lu2022single}. Secondly, depending on the starting point of optimization, called the initial solution, optimization can converge to different solutions, some of which can be highly suboptimal. Common approaches to this challenge are the random-sampling shooting method~\cite{robust-mpc} and modeling the problem as a convex problem~\cite{chen2019optimal}.

As an example of a real-world physical system, consider a scenario in which an energy company aims to optimize the production of its power plant over a time horizon $T$. The objective is to minimize operational cost while satisfying consumers' heat demand. In this work, we consider the following class of time-dependent optimization problems with inequality constraints
\begin{align}
\begin{split}
&\min_{\mathbf{u} \in U \subseteq \mathbb{R}^T} J(\mathbf{u}) \quad \\ &\text{s.t.} \quad f_i(\mathbf{u}_i) \geq q_i, \quad i=1,\ldots,T 
\end{split}
\tag{P}
\label{eq:constrained_optimization_problem}
\end{align}
where the feasible set $U$ is the Cartesian product of $T$ closed sets: $U = \prod_{i=1}^{T}U_i$ with $U_i \subseteq \mathbb{R}$. The decision variable $\mathbf{u} \in U$ is decomposed as $\mathbf{u}=(u_{1},\ldots,u_{T})$ with the decision variable at the time step $i$, $u_i \in U_i$, $i=1,\dots, T$. Incorporating a moving time window approach~\cite{bertsekas2012dynamic}, the decision variable $\mathbf{u}_i  \in U_i^{n_w+1}$, $U_i^{n_w+1} = \prod_{j=0}^{n_w} U_{i-j}$, consists of $n_w$ preceding decision variables and the current decision variable, $\mathbf{u}_i = (u_{i-n_w},\ldots,u_i)$. The objective function $J(\mathbf{u}): \mathbb{R}^T \rightarrow \mathbb{R}$ is defined as $J(\mathbf{u})=\sum_{i=1}^{T}J_i(u_i)$, with the objective function at each time step $i$, $J_i(u_i): \mathbb{R} \rightarrow \mathbb{R}$ being increasing in all decision variables. Functions $f_i(\mathbf{u}_i):\mathbb{R}^{n_w+1} \rightarrow \mathbb{R}$ are non-decreasing in all decision variables, and their values should be greater than or equal to the right-hand side of the constraints, constants $q_i \in \mathbb{R}$. All functions are (possibly) nonlinear and non-convex, and continuously differentiable. We also assume that solutions to the optimization problem exist when all inequality constraints are converted into equality constraints.

The defined minimization problem has an increasing objective function and non-decreasing inequality constraint functions in all decision variables. Under these properties and the above assumption, the optimal solution of the problem lies on the boundary of the feasible region, on equality constraints. Since we aim to minimize the objective function, any increase in the values of already satisfied inequality constraints will lead to an increase in the objective function value. However, our aim is not solely to find the optimal solution, but rather to attain a feasible solution with favorable objective function value within the time limit. For efficiently solving this minimization problem, we adapt the standard penalty-based method, described in Section~\ref{sec:standard_penalty}, by iteratively updating the right-hand side of the constraints using some guardrail variable, which is based on constraint values in earlier iterations. Furthermore, by initiating the optimization with a feasible initial solution, we achieve computationally tractable optimization.

The primary contributions of our work are:
\begin{itemize}
    \item Given the previously established monotonicity properties of the objective function and inequality constraints, we propose a penalty-based guardrail algorithm (PGA), with a relatively small penalty parameter, and a feasible initial solution. 
    \item We customize the PGA to two novel application domains inspired by the optimization of the district heating system (DHS): a hand-made simplified model of the DHS and a model derived by approximating realistic model of the DHS using deep neural networks, and compiling these learned networks into optimization.
    \item The PGA significantly outperforms mathematical programming solvers and the penalty-based method, and achieves better performance and faster convergence compared to the increasing penalty dual decomposition (IPDD) algorithm within a specified time limit.
\end{itemize}
\section{Related Work}
\label{sec:related_work}
Efficiently solving large-scale optimization problems has been a long standing question in optimization as well as in other fields. Existing works study different approaches, such as formulating the problem using the mathematical programming (MP) framework and then solving it with MP solvers, or transforming a constrained optimization problem into an unconstrained one using penalty-based or augmented Lagrangian functions, and then solving it with computationally efficient first-order approaches~\cite{khajavirad2018hybrid, luenberger2016penalty, bertsekas2012dynamic}.

\paragraph{Mathematical Programming Solvers.}Off-the-shelf MP solvers require low computational times for problems with specific structures, such as linear, convex, and/or small-scale problems~\cite{luenberger1984linear}. However, numerous engineering and energy systems, such as district heating systems, are nonlinear, non-convex and of a large scale. Using branch and bound~\cite{land2010automatic}, branch and reduce~\cite{ryoo1996branch}, cutting planes~\cite{kelley1960cutting}, and outer approximation~\cite{duran1986outer} approaches, several global deterministic solvers have been developed for these problems, for example: Gurobi~\cite{gurobi}, SCIP~\cite{10.1145/3585516} and BARON~\cite{nohra2021sdp}. The performance of these solvers strongly depends on an effective use of state-of-the-art methods for linear programming, mixed-integer programming, and convex programming, at various stages in the global search~\cite{khajavirad2018hybrid}. However, for large-scale problems, the generated search trees can become arbitrarily large, resulting in high computational times~\cite{molina2024differential}.
\paragraph{Penalty-based Methods.}
Penalty-based methods approximate constrained optimization problems by unconstrained ones~\cite{luenberger2016penalty}. They have been successfully applied to numerous real-world domains, such as power flow optimization~\cite{jabr2021penalty}, unit commitment~\cite{palani2021fast} and robot manipulation~\cite{gaz2019dynamic}. Finding an effective penalty function to serve as a surrogate for missing constraints, and determining the appropriate strengths for the penalty parameters, can be challenging. As noted in~\cite{siedlecki1993constrained}, much of the difficulty arises because the optimal solution lies on the boundary of the feasible region, which is also the case in our optimization problem. 
On one hand, if the penalty parameters are large, a method finds a feasible solution, but gets stuck in a poor local minimum, a solution that satisfies constraints but with a sub-optimal value of the objective function. Moreover, the penalty function can become ill-conditioned, with large gradients and abrupt function changes~\cite{bryan2005penalty}. On the other hand, if the penalty parameters are small, a method finds a deep local minimum, a solution that violates constraints with an optimistically low value of the objective function~\cite{platt1987constrained}. Therefore, choosing and updating the penalty parameters is challenging,
and significantly impacts the quality of the obtained solution.

\paragraph{Inexact Augmented Lagrangian Methods.}Augmented Lagrangian (AL) methods have been developed to address the ill-conditioning problem of penalty methods by augmenting standard Lagrange multiplier method with a penalty term. This method was firstly proposed by~\citeauthor{hestenes1969multiplier}~[\citeyear{hestenes1969multiplier}] and \citeauthor{powell1969method}~[\citeyear{powell1969method}] and studied in detail by~\citeauthor{bertsekas2014constrained} and \citeauthor{birgin2014practical} in \citeyear{bertsekas2014constrained}. For large-scale and stochastic problems, subproblems of an AL method are often solved inexactly, with truncated iterates of the method minimizing the AL function~\cite{fernandez2012local}. Because of convergence and computational complexity analysis, many existing works use AL methods for optimization problems with linear or convex constraints~\cite{kong2019complexity,xu2021iteration,zeng2022moreau}. Under some assumptions and conditions necessary for analyzing convergence of algorithms, such as weak convexity and smooth constraints, provably efficient AL methods have been proposed for non-convex/nonlinear optimization problems with equality and inequality constraints in~\cite{sahin2019inexact,ma2020quadratically,lu2022single,lin2022complexity,boob2023stochastic}. However, these assumptions are not satisfied for our real-world physical system. Adopting milder assumptions suitable for practical applications, the increasing penalty dual decomposition (IPDD) method has been proposed in~\cite{shi2020penaltyI} for optimization problems with a non-convex objective function and equality constraints. The following four steps describe this algorithm: 1) Optimizing the AL function; 2) Checking the feasibility of constraints; 3) Updating the dual variable based on constraint violations and the penalty parameter; 4) Increasing the penalty parameter and repeating from step one until certain convergence criteria are met. The IPDD method has been successfully applied to three real-world control problems from the signal processing domain~\cite{shi2020penaltyII}. While this method is computationally efficient, it might have higher computational time compared to a penalty-based method due to the inclusion of Lagrangian terms.
\paragraph{Making the optimization computationally tractable.}Transforming constrained optimization problems into unconstrained ones using the penalty-based and AL methods enables computationally efficient optimization with algorithms based on first-order information, such as the gradient descent method. However, depending on the initial solution, this method can converge to different solutions of wildly varying quality. To address this problem, a method called the random-sampling shooting method starts optimization from various initial solutions, and selects the solution with the best performance~\cite{robust-mpc}. However, this approach can be too time consuming. 
An alternative approach involves approximating the system's dynamics with computationally tractable models, such as input convex neural networks~\cite{chen2019optimal}. This approach has been successfully applied in minimizing energy consumption of buildings~\cite{bunning2021input}, optimizing a chemical reactor~\cite{yang2021optimization} and controlling an engine airpath system~\cite{moriyasu2021structured}.
\section{The Standard Penalty-Based Method}
\label{sec:standard_penalty}
Since the optimal solution of the constrained minimization problem (\ref{eq:constrained_optimization_problem}) with inequality constraints lies on the boundary of the feasible region, we first convert the inequality constraints into equality constraints. Then, we transform this problem into an unconstrained one using the standard penalty-based method, formally described in Algorithm~\ref{alg:penalty-based-method}. This method shares the challenges associated with penalty-based methods, as discussed in the previous section. 

To approximate constraints, a quadratic regularization term is added to the objective function:  
\begin{align}
\begin{split}
    &\min_{\mathbf{u} \in U \subseteq \mathbb{R}^T} \hat{J}(\mathbf{u}), \text{where} \\
    &\hat{J}(\mathbf{u}) = \sum_{i=1}^{T} \left(J_i(u_i)+ C\left(f_i(\mathbf{u}_i)-q_i\right)^2\right),
\end{split}
\end{align}
where $\hat{J}(\mathbf{u}):\mathbb{R}^{T} \rightarrow \mathbb{R}$ is the penalty function and positive $C \in \mathbb{R}^+$ denotes the penalty parameter. 

The penalty function is inexactly solved using the gradient descent method with learning rate $\alpha$, starting from an initial solution $\mathbf{u}^0 \in U$. If, at the end of any gradient descent iteration, the decision variable falls outside the feasible set $U$, it is projected back to the nearest point in this set using the projection function $P_U$ (line~\ref{line:P_U} of Algorithm~\ref{alg:penalty-based-method}). The optimization is repeated until the absolute difference between values of decision variables from $N$ subsequent gradient descent iterations is less than the threshold $\Delta$, which defines the gradient descent (GD) stopping criterion (line~\ref{line:GD_stopping_criterion}).
\begin{algorithm}[t!]
\label{alg:penalty-based-method}
 \caption{The standard penalty-based method}
  \KwIn{Penalty parameter $C \in \mathbb{R}^+$, learning rate $\alpha \in \mathbb{R}^+$.}
  \KwInit{Initial solution $\mathbf{u}^{0} \in U$.}
  $j \leftarrow 0$\\
  Define $\hat{J}(\mathbf{u}^j)$ to be $\sum_{i=1}^{T}\left(J_i(u^j_i)+C\left(f_i(\mathbf{u}^j_i)-q_i\right)^2\right)$ \\
  \Repeat{GD stopping criterion is met \label{line:GD_stopping_criterion}}{
  $j \leftarrow j+1$ \\
   $\mathbf{u}^{j} \leftarrow P_U\left(\mathbf{u}^{j-1} -\alpha \nabla \hat{J}\left(\mathbf{u}^{j-1}\right)\right)$ \label{line:P_U}
  }
  \KwOut{Solution $\mathbf{u}^{j} \in U$.} 
\end{algorithm}
\section{The Penalty-Based Guardrail Algorithm}
Finding the optimal solution of the constrained minimization problem (\ref{eq:constrained_optimization_problem}) by approximating this problem as an unconstrained one and solving it with the gradient descent method is challenging. Our aim is to attain a feasible solution with favorable objective function value within the time limit. To obtain such a solution, this section proposes the penalty-based guardrail algorithm (PGA). 

The PGA combines the ideas of the aforementioned standard penalty-based method and iteratively updating the penalty function $\hat{J}(\mathbf{u})$ to address constraint violations across two iterations. Its inner iteration inexactly solves the penalty function using the gradient descent method. Then, its outer iteration updates the right-hand side of constraints with non-negative guardrail variables which add a margin based on scaled constraint values. To guarantee asymptotic convergence of the PGA, we scale these constraint values using a factor inversely proportional to the number of outer iterations. Algorithm~\ref{alg:pgda} formally describes the PGA.

Before proceeding to the algorithm's theoretical characteristics, we elaborate on the choice concerning the strength of the penalty parameter $C$. If the strength is large, then the solution of the inner iteration resides on or near equality constraints ($f_i(\mathbf{u}_i)-q_i \approx 0$, $i=1,\ldots,T$ in line~\ref{line:gamma} of Algorithm~\ref{alg:pgda}), but it might get trapped in a poor local minimum. However, then the update to the guardrail variable (line~\ref{line:guardrail}) has no effect. As a result, the optimization process will remain in a poor local minimum. To avoid this, we choose a small strength for the parameter $C$ across all application domains. Next, we describe and analyze theoretical properties of the PGA.
\begin{algorithm}[t!]
 \caption{The penalty-based guardrail algorithm}
 \label{alg:pgda}
  \KwIn{Penalty parameter $C \in \mathbb{R}^+$, learning rate $\alpha \in \mathbb{R}^+$.}
  \KwInit{Guardrail variable $\epsilon_i^0 \leftarrow 0$, $i=1,\ldots,T$, initial solution $\mathbf{u}^{0,0} \in U$.}
  $k \leftarrow 0$\\
  \Repeat{time limit stopping criterion is met}{
  $j \leftarrow 0$\\
  Define $\hat{J}(\mathbf{u}^{k,j})$ to be $\sum_{i=1}^{T}\left(J_i(u^{k,j}_i)+C\left(f_i(\mathbf{u}^{k,j}_i)-q_i-\epsilon^k_i\right)^2\right)$ \\
  \Repeat{GD stopping criterion is met}{ 
   $j \leftarrow j+1$\\
   $\mathbf{u}^{k, j} \leftarrow P_U\left(\mathbf{u}^{k, j-1} - \alpha\nabla \hat{J}\left(\mathbf{u}^{k,j-1}\right)\right)$\\
  }
  $k \leftarrow k+1$ \\
  \For{$i=1,\ldots,T$}{
  $\gamma^{k-1}_i \leftarrow f_i(\mathbf{u}^{k-1,j}_i)-q_i$ \label{line:gamma} \\
  $\epsilon^{k}_i \leftarrow \max(0, \epsilon^{k-1}_i - \frac{1}{k}\gamma^{k-1}_i)$ \label{line:guardrail}
  }
  }
  \KwOut{Solution $\mathbf{u}^{k,j} \in U$.}
\end{algorithm}

First, we focus on achieving computationally tractable optimization with the gradient descent method, aiming for convergence to relatively similar solutions. To facilitate this optimization, we construct non-decreasing penalty function $\hat{J}(\mathbf{u})$. For a decision variable $\mathbf{u}_i$, the following proposition holds:
\begin{proposition}
At any decision variable $\mathbf{u}_i \in U$, satisfying $f_i(\mathbf{u}_i) - q_i \geq 0$, $i=1,\ldots,T$, the penalty function $\hat{J}(\mathbf{u})$ is non-decreasing in the decision variables.
\end{proposition}
The proof follows from the fact that the composition of non-negative increasing and non-decreasing functions is also a non-decreasing function, and that the non-negative sum of increasing and non-decreasing functions is also a non-decreasing function~\cite{boyd2004convex}. Therefore, we initiate the optimization in Algorithm~\ref{alg:pgda} with a feasible initial solution $\mathbf{u}^0 \in U$, where the penalty function $\hat{J}(\mathbf{u})$ is ensured to be non-decreasing. Due to this property of the function, we hypothesize that the optimization using gradient descent method will converge to relatively similar solutions, regardless of the chosen feasible initial solution. This hypothesis is later evaluated empirically.

During gradient descent optimization in the inner loop of the PGA, the penalty function remains non-decreasing until a point where the first inequality constraint is violated, $f_i(\mathbf{u}_i)-q_i<0$. The following proposition states an outcome at the minimum of the penalty function:  

\begin{proposition}
At any minimum of the penalty function $\hat{J}(\mathbf{u})$, at least one inequality constraint $f_i(\mathbf{u}_i)-q_i \geq 0$, $i=1,\ldots,T$ will be violated.    
\label{prop:infeasibility}
\end{proposition}

\begin{proof}
The gradient of the penalty function is zero at any minimum $u_i^*$:
\begin{align}
\begin{split}
& \left.\frac{\partial \hat{J}(\mathbf{u})}{\partial u_i} \right|_{u_i = u^*_i} = 0 \Rightarrow \\ & \sum_{i=1}^{T} (f_i(\mathbf{u}^*_i)-q_i) \left. \frac {\partial f_i(\mathbf{u}_i)}{\partial u_i}\right|_{u_i = u^*_i} = - \frac{1}{2C} \left.\frac{\partial J(\mathbf{u})}{\partial u_i} \right|_{u_i = u^*_i}
\end{split}
\end{align}
As the function $f_{T}(\mathbf{u}_{T})$ is the only one dependent on the decision variable at the last time step $u_{T}$, we examine the gradient of the penalty function with respect to this variable:
\begin{equation}
\resizebox{.91\linewidth}{!}{$
            \displaystyle
    (f_{T}(\mathbf{u}_{T}^*)-q_{T})\left. \frac {\partial f_{T}(\mathbf{u}_{T})}{\partial u_{T}}\right|_{u_{T}= u^*_{T}}=- \frac{1}{2C} \left.\frac{\partial J(\mathbf{u})}{\partial u_{T}} \right|_{u_{T} = u^*_{T}}
    $}
\end{equation}
Following the fact that the gradient of an increasing function is positive $\frac{\partial J(\mathbf{u})}{\partial u_{T}}>0$, the gradient of a non-decreasing function is non-negative $\frac {\partial f_{T}(\mathbf{u}_{T})}{\partial u_{T}} \geq 0$, and $C>0$, it follows that the constraint at the last time step $f_{T}(\mathbf{u}_{T}^*)-q_{T}<0$ will be violated.
\end{proof}
To address constraint violations, we propose updating the right-hand side of the constraints with non-negative guardrail variables $\epsilon_i \in \mathbb{R}_{\geq 0}$ for each $i=1,\ldots,T$ in the outer loop of the PGA, as defined in Definition~\ref{def:guardrail}. The Proposition~\ref{prop:guardrail} highlights the effect of this guardrail variable on optimized decision variables $\mathbf{u}$.
\begin{definition}
Changing the right-hand side of the constraints $q'_i = q_i + \epsilon_i$, $i=1,\ldots,T$, in the penalty function $\hat{J}(\mathbf{u})$, where $\epsilon_i$ is a guardrail variable, yields another penalty function $\hat{J}'(\mathbf{u})$.
\label{def:guardrail}
\end{definition}
\begin{proposition}
At any minimum of the penalty function $\hat{J}(\mathbf{u})$, the first step of the gradient descent method on $\hat{J}'(\mathbf{u})$ will not decrease any $u_i$, $i=1,\ldots,T$.
\label{prop:guardrail}
\end{proposition}
\begin{proof}
First, we reformulate the changed penalty function $\hat{J}'(\mathbf{u})$:
\begin{equation}
 \hat{J}'(\mathbf{u}) = \sum_{i=1}^{T}\left(J_i(u_i)+C\left(f_i(\mathbf{u}_i)-q_i-\epsilon_i\right)^2\right) 
\end{equation}
Upon expanding the term $(f_i(\mathbf{u}_i)-q_i-\epsilon_i)^2$ and subsequently re-arranging the components of the equation, we arrive at the following expression:
\begin{equation}
\hat{J}'(\mathbf{u}) =\hat{J}(\mathbf{u})+C\sum_{i=1}^{T}\left(\epsilon_i^2 - 2\epsilon_i\left(f_i(\mathbf{u}_i)-q_i\right)\right)
\end{equation}
At any minimum $\mathbf{u}^*$ of the penalty function $\hat{J}(\mathbf{u})$ holds $\nabla \hat{J}(\mathbf{u}^*)=\mathbf{0}$. Using the definition that the gradients of a non-decreasing function are non-negative, $\nabla f_i (\mathbf{u}_i) \geq \mathbf{0}$, and $\epsilon_i \geq 0$, we derive the following conclusion regarding the sign of the gradients of the penalty function $\hat{J}'(\mathbf{u})$:
\begin{equation}
\nabla \hat{J}'(\mathbf{u}^*) = \underbrace{\nabla \hat{J}(\mathbf{u}^*)}_{=\mathbf{0}}-2C\sum_{i=1}^{T} \underbrace{\epsilon_i}_{\geq 0}\underbrace{\nabla f_i(\mathbf{u}^*_i)}_{\geq \mathbf{0}} \leq \mathbf{0}    
\end{equation}
Therefore, gradient descent moves the optimized parameters $\mathbf{u}$ in the direction of the negative gradient, so no $u_i$ can decrease.
\end{proof}
If the inequality constraints in the outer iteration $k$ are violated by a margin $\gamma^k_i = f_i(\mathbf{u}^k_i)-q_i$ (line~\ref{line:gamma} of Algorithm~\ref{alg:pgda}), we increase the guardrail variable in proportion to the scaled
constraint violations $\epsilon^{k+1}_i \leftarrow \epsilon^k_i - \frac{1}{k+1}\gamma^k_i$ (line~\ref{line:guardrail}), and repeat optimization using gradient descent. According to the Proposition~\ref{prop:guardrail}, updating the guardrail variable will lead to no decrease of decision variables at the minimum of the original penalty function $\hat{J}(\mathbf{u})$. Therefore, given that the functions $f_i(\mathbf{u}_i)$
are non-decreasing, their values will either increase or in the worst case, remain constant.

However, the PGA has a caveat regarding theoretical convergence to a feasible solution. If $\gamma_i^k > 0$ in the $k$th outer iteration, we might decrease the guardrail variable in the subsequent $(k+1)$th iteration following the update rule $\epsilon^{k+1}_i \leftarrow  \max (0, \epsilon^k_i - \frac{1}{k+1}\gamma^k_i)$. This adjustment means that the guardrail variable in the $(k+1)$th iteration might be smaller than in the $k$th iteration. Consequently, the decision variable at the minimum of the penalty function in the $k$th iteration $\hat{J}(\mathbf{u}^k)$ could decrease. We empirically evaluate the convergence of the PGA for the time limit stopping criterion. 

\section{Application Domains}
To evaluate PGA on a complex real-life optimization problem, we present two completely different models for the optimization of a district heating system (DHS): a hand-made simplified model of the DHS, and a model derived by approximating the DHS using deep neural networks (DNNs), which is subsequently compiled into an optimization model.

The hand-made simplified model of the DHS makes significant, expert-derived assumptions on the realistic model of the DHS. Together these assumptions ensure meeting the monotonicity conditions
for the PGA, and computationally stable optimization using MP solvers. The next domain approximates a realistic but computationally very hard model of the DHS~\cite{li2015combined} using specific DNNs, and derives an optimization model by compiling learned networks into optimization. These DNNs facilitate monotonicity properties of the optimization problem, and computationally stable optimization using MP solvers. First, we elaborate on the realistic model of the DHS, followed by both application domains for the PGA.

\subsection{A Realistic Model of District Heating System}
\label{sec:dhs_model}
District heating systems play a crucial role in the transition to smart energy systems because of the flexibility in their electricity use~\cite{persson2019heat}. An important source of flexibility in a DHS originates from the thermal inertia of the district heating network itself, called pipeline energy storage. This flexibility can decouple heat and electricity production of the combined heat and power (CHP) production plant to minimize operational cost, while meeting the consumer's heat demand~\cite{averfalk2020economic}.

\begin{figure}[t!]
    \centering
    \includegraphics[width=1\linewidth]{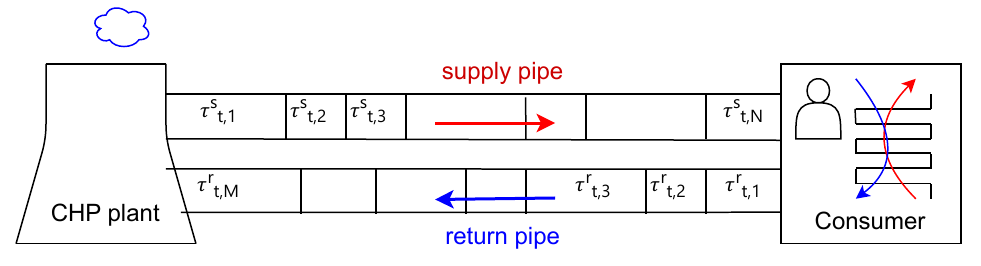}
    \caption{Structure of a district heating system.}
    \vspace*{7mm}
    \label{fig:dhs}
\end{figure}
Figure~\ref{fig:dhs} illustrates the structure of a DHS. The system consists of a CHP plant, a single consumer and a district heating network. At time step $i$ the CHP plant simultaneously produces heat $h_i$ and electricity $p_i$. This heat is used to raise the temperature of the water to a specific supply temperature $\tau^s_i$, which is then transported to the consumer via a supply pipe. The delivered heat at the end of the supply pipe should be equal to or greater than the consumer's heat demand $q_i$. After consumption, the water is cooled down and returned to the CHP plant via the return pipe. As a district heating network is usually few kilometers long, there is a delay between an increase of supply temperature at the producer and the corresponding increase of supply temperature at the consumer. This delay results in a pipeline energy storage, which charges when the supply temperature is raised and discharges when the supply temperature is lowered~\cite{merkert2019optimal}.

The aim is to minimize operational cost while meeting the consumer's heat demand by planning the production of heat and electricity at the CHP plant over an optimization horizon 
$T$ and leveraging this pipeline energy storage. Due to  rapidly changing external conditions, such as fluctuations in electricity prices, the planning should be conducted in real-time. This optimization problem corresponds to the optimization problem (\ref{eq:constrained_optimization_problem}), with the CHP operation region represented by the feasible set $U$, and the produced heat and electricity the decision variable $\mathbf{u}=(h_{1}, p_{1},\ldots, h_{T}, p_{T})$. The objective is the minimization of operational cost, i.e., the sum of the cost $a_0 h_i$ and $a_1 p_i$ for producing heat and electricity, respectively, over optimization horizon $T$: $J(h_{1}, p_{1},\ldots,h_{T},p_{T}) = \sum_{i=1}^{T} a_0 h_i+a_1p_i$. Finally, the functions $f_i$, which are highly nonlinear and non-convex, model the delivered heat based on previous and current produced heats. Their values at each time step $i$ should be equal to or greater than the consumer's heat demand $q_i$, $f_i \geq q_i$. The detailed model of the realistic DHS can be found in~\cite{li2015combined}, with the key information available in Appendix A.

The functions $f_i$ are not non-decreasing in produced heat decision variables. Therefore, the PGA can not efficiently solve the optimization problem. Moreover, the complexity of the functions $f_i$ poses challenges for computationally stable optimization using MP solvers. The following two models address these challenges differently.  
\subsection{A Simplified Model of District Heating System}
To simplify the realistic model of the DHS and reformulate functions $f_i$ into a non-decreasing form in produced heats, we assume a constant and zero ambient temperature and a zero temperature in the return pipe. While these assumptions simplify the realistic model, they still capture the time delays in the supply pipe. Then, we transform the constrained optimization problem into an unconstrained by constructing the penalty function $\hat{J}(h_{1}, p_{1},\ldots, h_{T}, p_{T})$. Detailed information about this transformation process is available in  Appendix B.

\subsection{An Optimization Model Derived from Deep Neural Networks}
In this section, we build DNNs to approximate parts of the realistic model of the DHS. Then, we compile these learned DNNs into the optimization model. This process consists of five steps: defining the structure of the optimization problem, obtaining a training dataset, learning approximate models using DNNs, constructing a mathematical model, and optimizing the model.

Firstly, to define the structure of the optimization problem, we identify which complex relations in the realistic problem to approximate using DNNs. In the DHS, this relation represents the delivered heat to the consumer, and we approximate it using the state transition function, $s_i=g(s_{i-1}, h_{i-n_w},p_{i-n_w},\ldots,h_i,p_i)$ and the system output function $y_i = f(s_i,  h_{i-n_w},p_{i-n_w},\ldots,h_i,p_i)$, $i=1,\ldots,T$, where $s_i$ stands for the state of the system at time step $i$ and $y_i$ is the delivered heat to the consumer. The system state $s_i$ consists of observations of the temperature at the inlet of the supply pipe, the temperature at the outlet of the supply pipe, and the mass flow. The defined state is the partially observable state of the DHS, and it is the approximation of the believed distribution over the state space. 

Secondly, we collect a training dataset consisting of inputs and outputs of the functions $f$ and $g$ using a simulation environment~\cite{wu2022gridpenguin}. Thirdly, we build DNNs to approximate state transition and system output functions from obtained training dataset. To model non-decreasing constraint functions, these neural networks should be non-decreasing in all decision variables. For this, we constrain their input weights to be non-negative in all decision variables, all weights between layers to be non-negative and use non-decreasing activation functions, such as rectified linear units. The size of these neural networks is particularly important since it affects the modeling capability and the computational stability of the optimization with MP solvers. For example, while larger neural networks have higher representation power, their increased complexity may comprise the computational stability of optimization.

Fourthly, we chain these fixed learned deep neural network models together for a set planning horizon $T$. The integration of these models into the optimization framework depends on the type of solver. If a MP solver is used, learned DNNs are compiled into a mixed integer linear program following framework defined in~\cite{fischetti2018deep}. If a first-order solver is used, the learned models are compiled in a recurrent neural network unrolled for the fixed planning horizon. Given objective function, defined in Section~\ref{sec:dhs_model} as $J=\sum_{i=1}^{T}a_0h_i+a_1p_i$, we construct mathematical model. Lastly, for the optimization of this model using the PGA, we formulate the penalty function  $\hat{J}(h_{1}, p_{1},\ldots, h_{T}, p_{T})$. The details of these five steps are given in the Appendix C.

While DNNs approximate the DHS in our application, they are general function approximators. As long as the architecture of the DNNs adheres to the restrictions outlined in the third step, they can be integrated into the optimization using the PGA, independently of the underlying system they approximate. This approach can thus be used for many other realistic optimization problems without requiring deep understanding of the domain to decide how to model it using non-decreasing functions. 
\section{Experimental Design}
To evaluate the proposed PGA, we compare it with a mathematical programming (MP) solver, the penalty method (PM) and the increasing penalty dual decomposition (IPDD) method across one artificial domain and two domains inspired by the DHS, as described in the previous section. Based on the characteristics of the optimization problem, the MP solver is a nonlinear programming (NLP) solver for the first and second domains and a mixed integer linear programming (MILP) solver for the third domain.

All the experiments are performed on a computer with 4-core Intel I7 8665 CPU. Concerning MP solvers, used NLP solver is SCIP 1.11.1~\cite{10.1145/3585516}, and MILP solver is Gurobi 9.5.2~\cite{gurobi}.  For the first-order solvers implemented in the PM, the IPDD method, and the PGA, we use Tensorflow version 2.13~\cite{tensorflow2015-whitepaper}. 

The artificial application domain is a basic  N-dimensional optimization problem with three inequality constraints, described in Appendix D. Concerning two domains inspired by the DHS,  the optimization is conducted over a planning horizon of 12 hours, $T=12h$. The heat and electricity price coefficients $a_0$ and $a_1$ are taken from~\cite{abdollahi2014optimization}, and their values are $a_0 = 8.1817 [\euro/h]$ and $a_1 = 38.1805 [\euro/h]$. The heat demand data is obtained from the dataset provided in~\cite{ruhnau2019time}. Further details on parameters of the DHS are available in Appendix A.

For all application domains, the penalty functions are optimized using the Adam optimizer~\cite{kingma2017adam} with a learning rate 0.01. The gradient descent stopping criterion parameters $N$ and $\Delta$ are determined empirically for each application domain, and their values are listed in Table~\ref{table:optimization_parameters}. The same parameters are used for the PM, the IPDD method, and the PGA. Appendix E provides further details on the experiments.
\begin{table}[t!]
\centering
\caption{Optimizer parameters for solving the penalty function.}
\vspace*{7mm}
\begin{tabular}{lrrr}
\toprule
Domain &  N & $\Delta$\\
\midrule
N-dimensional problem & $50$&$10^{-6}$\\
District heating system & $1000$&$10^{-1}$\\
Deep neural networks &  $1000$&$10^{-1}$\\
\bottomrule
\end{tabular}
\label{table:optimization_parameters}
\end{table}
\section{Numerical Experiments}
\begin{figure*}[t!]
\centering
\subfigure[Objective function]{
  \includegraphics[width=0.42\linewidth]{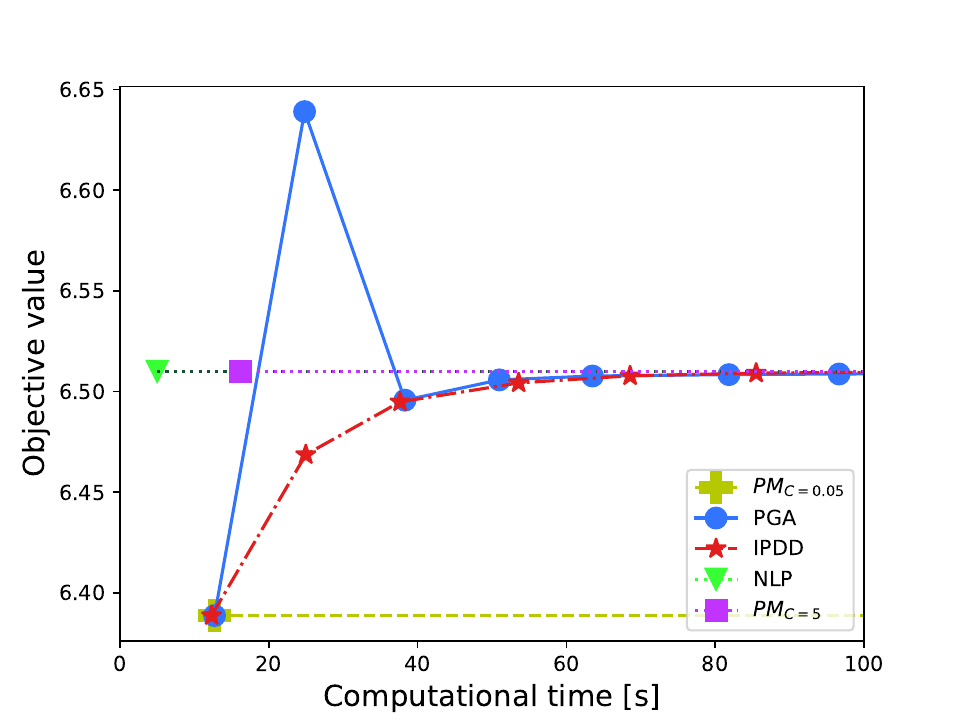}}
\subfigure[Infeasibility]{
  \includegraphics[width=0.42\linewidth]{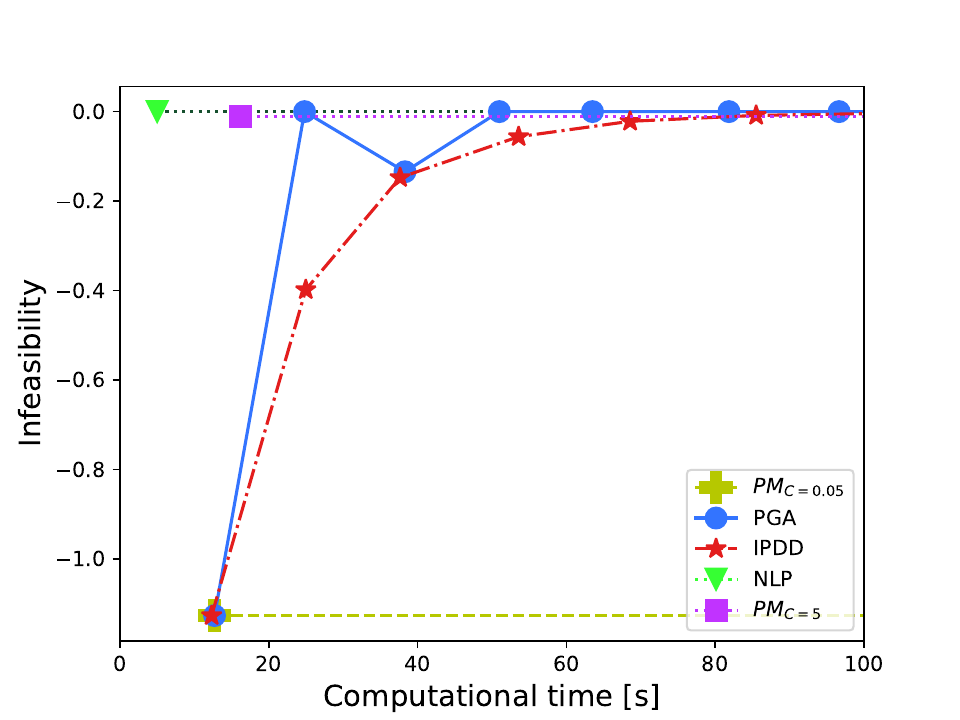}}
\caption{Comparison between the PGA, and the NLP, PM and IPDD methods for solving an N-dimensional optimization problem.}
\label{fig:toy_problem}
\end{figure*}

\begin{figure*}[t!]
\centering
\subfigure[Objective function]{
  \includegraphics[width=0.42\linewidth]{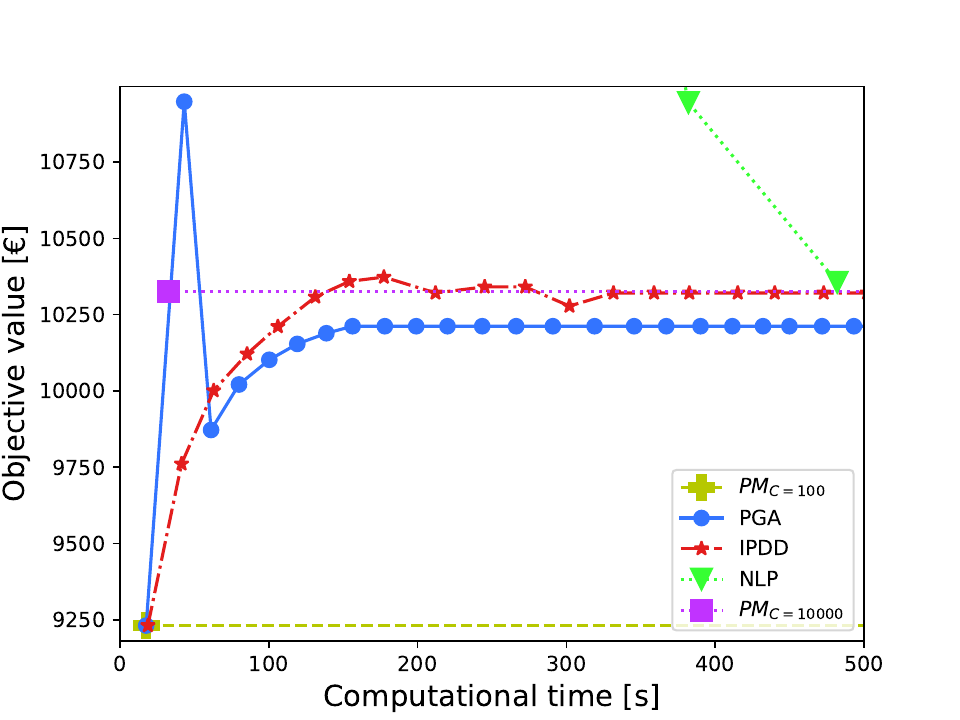}}
\subfigure[Infeasibility]{
  \includegraphics[width=0.42\linewidth]{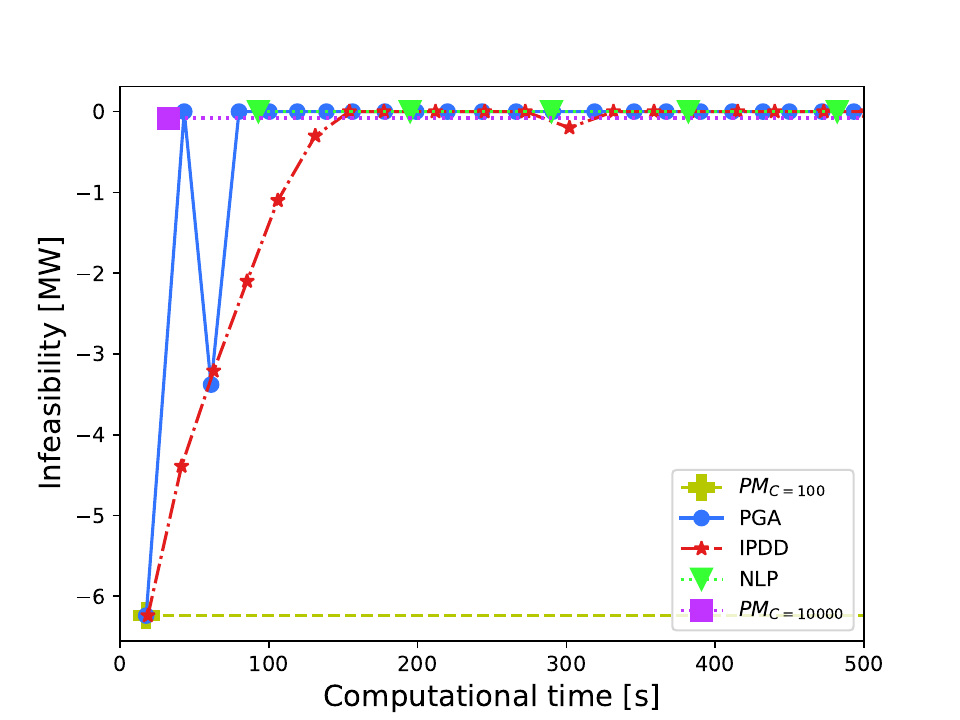}}
\caption{Comparison between the PGA, and the NLP, PM and IPDD methods for solving a simplified model of a district heating system.}
\label{fig:simplified_dhs}
\end{figure*}

\begin{figure*}[t!]
\centering
\subfigure[Objective function]{
  \includegraphics[width=0.42\linewidth]{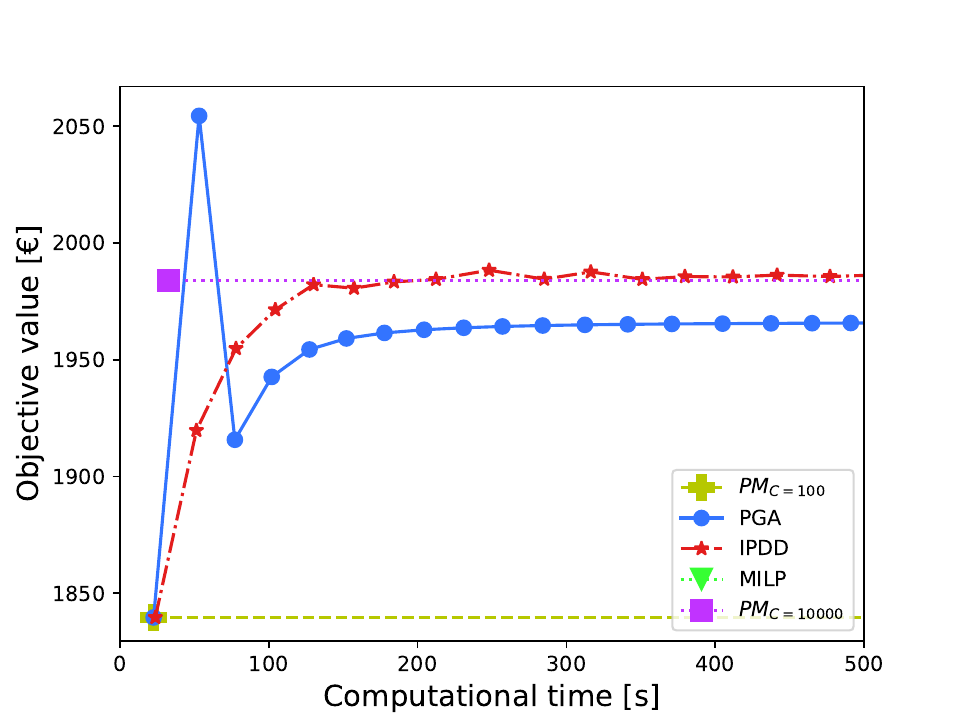}}
\subfigure[Infeasibility]{
  \includegraphics[width=0.42\linewidth]{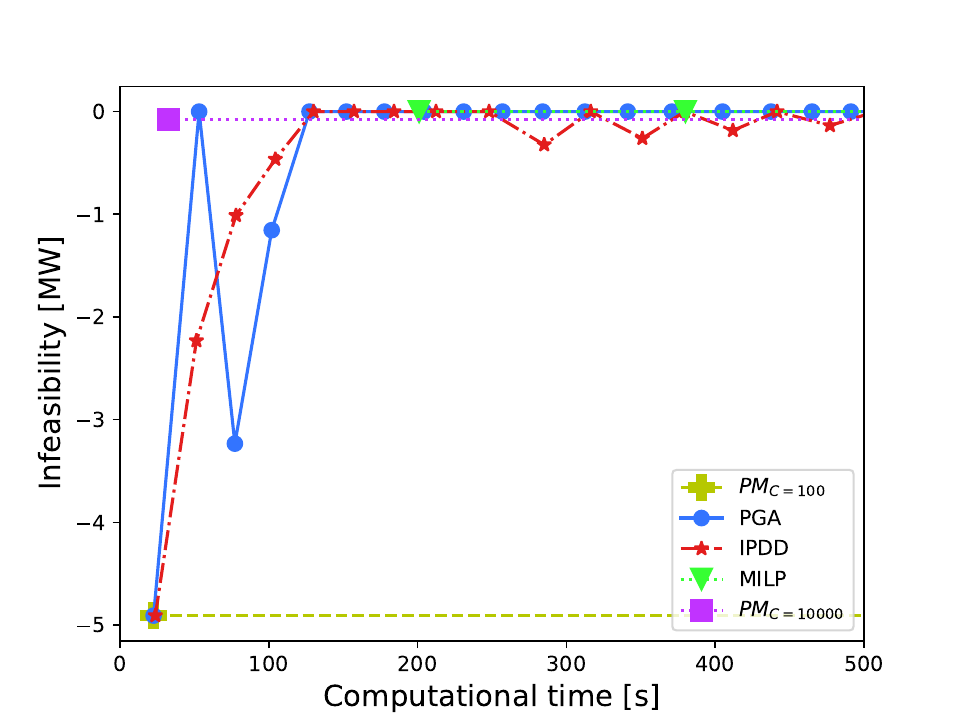}}
\caption{Comparison between the PGA, and the MILP, PM and IPDD methods for solving an optimization model derived from learned deep neural networks.}
\vspace*{7mm}
\label{fig:neural_network}
\end{figure*}
This section compares the PGA with a MP solver, the PM with two strengths of the penalty parameter and the IPDD method across the artificial domain and two domains inspired by the DHS. The IPDD method, briefly described in Section~\ref{sec:related_work}, has been successfully applied to a wide range of non-convex constrained problems~\cite{shi2020penaltyI, shi2020penaltyII}. The source code is available at~\cite{PGA}. 

Figure~\ref{fig:toy_problem} visualizes the performance of these methods on a basic N-dimensional optimization problem. In Figure~\ref{fig:simplified_dhs}, we compare the methods on a simplified model of a DHS for a day in the winter season, and in Figure~\ref{fig:neural_network} we examine their effectiveness on a model derived from DNNs for a day in the early spring season.

The $x$-axis in all figures represents the computational time of each algorithm. In the first column, the $y$-axis shows the objective value. In the second column, the $y$-axis represents infeasibility -- the value of the violated constraint $f_i(\mathbf{u}_i)-q_i$ with the largest absolute value (i.e., $\max_{i=1,\ldots,T} |f_i(\mathbf{u}_i)-q_i|$). When the solution is feasible, the value on the $y$-axis is zero. In the figures, the markers for the PGA and IPDD methods indicate their solutions at the end of each outer iteration. For the NLP and MILP solvers, the markers represent instances of their feasible solutions during the search process. 

It can be observed that for small-scale models, such as an N-dimensional optimization problem, the MP solver finds the optimal solution within seconds. However, in the case of large-scale models, such as two realistic application domains, and within a specified time limit, feasible solutions of the first-order methods, the PGA and IPDD, have lower operational costs compared to the feasible solution of the MP solver. For instance, in the application domain involving models derived from DNNs, feasible solutions of the MILP solver in the objective function plot exhibit high values, causing them to exceed the plot's bounds.

We analyze the standard PM with two strengths of the penalty parameter $C$. On one hand, if the penalty parameter is small, the method converges to a deep local minimum, i.e., an infeasible solution where the value of the objective function is optimistically low. On the other hand, if the penalty parameter is large, the method converges to a solution with an objective function value similar to that of IPDD. However, even with a large penalty parameter, the solution remains infeasible, as it slightly violates at least one constraint. This observation aligns with Proposition~\ref{prop:infeasibility}. The analysis of Figure~\ref{fig:toy_problem} shows that both PGA and IPDD methods find a feasible solution with similar objective function values. However, as visualized in Figures~\ref{fig:simplified_dhs} and~\ref{fig:neural_network}, PGA finds a feasible solution with a lower objective function value compared to the IPDD.

Additionally, in the second column in Figures~\ref{fig:toy_problem},~\ref{fig:simplified_dhs} and~\ref{fig:neural_network} it can be observed that the PGA converges faster to the feasible solution compared to the IPDD method. As the IPDD is an augmented Lagrangian method, the augmented Lagrangian function in the IPDD includes $N$ additional linear terms compared to the penalty function in the PGA, with $N$ representing the number of constraints. In Table~\ref{table:computational_time}, we show that the computational time required for a single outer iteration of the PGA is lower compared to the IPDD across all three application domains. Also, as the complexity of the domain increases, this difference becomes more pronounced.
\begin{table}[t!]
\centering
\caption{Mean and standard deviation of computational time per single outer iteration.}
\vspace*{7mm}
\begin{tabular}{lrr}
\toprule
Domain & PDA [s] & IPDD [s]\\
\midrule
N-dimensional problem &  $13.89 \pm 2.04$&$14.96 \pm 2.60$\\
District heating system & $21.38 \pm 2.81$ & $26.42 \pm 4.71$ \\
Deep neural networks & $27.78 \pm 2.65$&$34.55 \pm 6.73$ \\
\bottomrule
\end{tabular}
\label{table:computational_time}
\end{table}
\begin{table}[t!]
\centering
\caption{Maximum normalized Euclidean distance (ED) between different solutions and 1\% of the normalized difference between maximum $u_{\textit{max}}$ and minimum $u_{\textit{min}}$ values within the feasible region.}
\vspace*{7mm}
\begin{tabular}{lrr}
\toprule
Domain&ED&$0.01(u_{\textit{max}}-u_{\textit{min}})$\\
\midrule
N-dimensional problem & $10^{-5}$ &$10^{-3}$\\
District heating system & $10^{-6}$& $10^{-5}$ \\
Deep neural networks & $10^{-6}$&$10^{-4}$ \\
\bottomrule
\end{tabular}
\label{table:euclidean_distance}
\end{table}

To validate the computational tractability of the PGA, we initiate an optimization with twenty different feasible initial solutions for each domain. In Table~\ref{table:euclidean_distance}, we show the maximum Euclidean distance between the solutions to which optimization converged, normalized by the lowest objective value in the initial solution, and 1\% of the difference between maximum $u_{\textit{max}}$ and minimum $u_{\textit{min}}$ values within the feasible region $U$, normalized by the lowest objective value found in a feasible solution. This indicates that the optimization consistently converges to relatively similar solutions.

\section{Conclusion}
In this work, we propose the penalty-based guardrail algorithm (PGA) for solving minimization problems with increasing (possibly) nonlinear and non-convex objective function and non-decreasing (possibly) nonlinear and non-convex inequality constraints. The PGA finds feasible solution with favorable value of the objective function, while facilitating computationally efficient and tractable optimization, particularly important in real-world physical systems. On one artificial and two novel application domains inspired by the district heating system, the PGA significantly outperforms mathematical programming
solvers and the standard penalty-based method, and
achieves better performance and faster convergence compared to the IPDD algorithm.
\section{Acknowledgments}
This work was executed with a Topsector Energy Grant from the Ministry of Economic affairs of The Netherlands.
\bibliography{mybibfile}
\clearpage
\section*{Appendix A}
This section provides details on the realistic model of the DHS. This model consists of a combined heat and power (CHP) production plant, a single consumer, a supply pipe and a return pipe. 

At time step $i$, $i=1,\ldots,T$, the CHP plant simultaneously produces heat $h_i$ and electricity $p_i$ according to the feasible operational region $U$ in Figure~\ref{fig:chp_region}, defined in~\cite{abdollahi2014optimization}.

The objective is the minimization of operational cost, i.e., the sum of the cost $a_0 h_i$ and $a_1 p_i$ for producing heat and electricity, respectively, over optimization horizon $T$:
 \begin{equation}
 J = \sum_{i=1}^{T} a_0 h_i+a_1p_i.
 \label{eq:objective_function}
 \end{equation}

The produced heat is proportional to a water heat capacity $c$, a mass flow $\dot{m}_{s,i}$, and a temperature difference between the supply pipe $\tau^{\textit{in}}_{s,i}$ and the return pipe $\tau^\textit{out}_{r,i}$ at the producer side:
\begin{figure}[b!]
    \centering
    \includegraphics[width=1\linewidth]{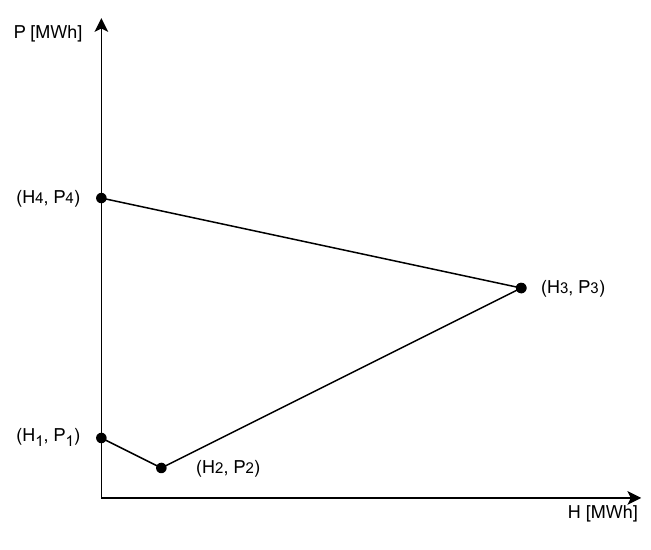}
    \caption{A combined heat and power plant produces heat (horizontal axis) and power (vertical axis) within the region defined by the four points.}
    \label{fig:chp_region}
    \vspace*{7mm}
\end{figure}
\begin{equation}
    h_i = c \dot{m}_{s,i} (\tau^{\textit{in}}_{s,i}-\tau^\textit{out}_{r,i})
\label{eq:produced_heat}
\end{equation}
Similarly, a delivered heat to the consumer $y_i$ is equal to the product of the water heat capacity, the mass flow, and a temperature difference between the supply $\tau_\textit{s,i}^\textit{out}$ and the return $\tau_\textit{r,i}^\textit{in}$ pipe. This delivered heat is modeled using functions $f_i(.)$, and defined as:
\begin{equation}
    y_i = c \dot{m}_{s,i} (\tau_\textit{s,i}^\textit{out}-\tau_\textit{r,i}^\textit{in})
\end{equation}
The delivered heat should be greater than or equal to a consumer's heat demand $q_i$:
\begin{equation}
    y_i \geq q_i
\end{equation}
The temperature changes at the pipe's outlet propagate gradually, with the outlet temperature being influenced by previous inlet temperatures and mass flows. Moreover, heat loss occurs due to the temperature difference between the water in the pipe and its surroundings, resulting in a decrease in temperature. These processes are time-dependent, nonlinear, and non-convex.

According to the node method proposed in~\cite{dhs-benonysson}, the temperature at the outlet of the pipe before the heat loss $\tau_{s,i}^{'\textit{out}}$ is defined as:

\begin{align}
\begin{split}
    & \resizebox{1\linewidth}{!}{$\tau_{s,i}^{'\textit{out}} = \frac{1}{\dot{m}_{s,i} \Delta t} \Bigg( \left(R_i-\rho A L \right) \tau_{s,i-\gamma_i}^{\textit{in}}+ \sum_{k=i-n_{w,i}+1}^{i-\gamma_i-1} \left( \dot{m}_{s,k} \Delta t \tau_{s,k}^\textit{in} \right)+$}\\
    &+ \left( \dot{m}_{s,i} \Delta t + \rho A L - S_i \right)\tau_{s,i-n_{w,i}}^\textit{in} \Bigg),
    \end{split}
\label{eq:no_temp_loss}
\end{align}
where $\Delta t$ denotes the time interval of the optimization, $\rho$ is the water density, $A$ is the surface area of the pipe and $L$ is the pipe's length. The integer variables $n_{w,i}$ and $\gamma_i$ represent the time delays in water propagation to the pipe's outlet, and variables $S_i$ and $R_i$ represent water masses flowing from the pipe's inlet to the pipe's outlet. These variables are intuitively illustrated in Figure~\ref{fig:water_flow}, adapted from~\cite{li2015combined}.

\begin{figure}[t!]
    \centering
    \includegraphics[width=1\linewidth]{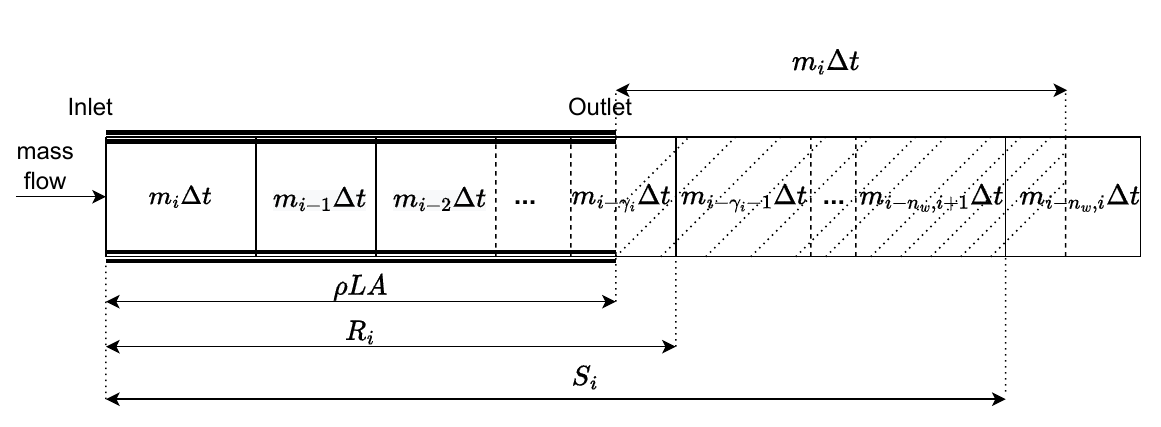}
    \caption{Water flow dynamics in the pipe of a district heating network.}
    \vspace*{7mm}
    \label{fig:water_flow}
\end{figure}
The term $i-\gamma_i$ indicates the index of the last period during which water mass outﬂows the pipe before the end of period $i$. This time delay is defined as:
\begin{equation}
    \resizebox{0.93\linewidth}{!}{$\gamma_{i} = \min_n \Biggl\{ n: \textit{s.t.} \sum_{k=0}^n \left( \dot{m}_{s,i} \Delta t \right) \geq \rho A L, n \geq 0, n \in \mathbb{Z} \Biggl\}$}
\end{equation}
 The term $i-n_{w, i}$ stands for the index of the last period during which water mass outﬂows the pipe before the end of period $i-1$, and it is defined as:
\begin{equation}
    \resizebox{0.93\linewidth}{!}{$n_{w,i} = \min_m \Biggl\{ m: \textit{s.t.} \sum_{k=1}^m \left( \dot{m}_{s,i} \Delta t \right) \geq \rho A L, m \geq 0, m \in \mathbb{Z} \Biggl\}$}
\end{equation}
The variable $R_i$ represents the mass of water flowing from period $i-\gamma_i$ to $i$:
\begin{equation}
    R_i = \sum_{k=0}^{\gamma_i} \left (\dot{m}_{s,i-k} \Delta t \right )
\end{equation}
Similarly, the variable $S_i$ represents the mass of water flowing from period $i-n_{w,i}$ to $i$:
\begin{equation}
    S_i = \begin{cases}
        \sum_{k=0}^{n_{w,i}-1} \left ( \dot{m}_{s,i-k} \Delta t \right )\,& \text{if } n_{w,i} \geq \gamma_i+1 \\
        R_i,& \text{otherwise.}
    \end{cases}
\end{equation}
Accounting for the heat loss, the temperature at the outlet of the pipe is determined as:
\begin{equation}
     \resizebox{0.93\linewidth}{!}{$\tau^\textit{out}_{s,i} = \tau_i^\textit{am} +(\tau_{s,i}^{'\textit{out}} - \tau_i^\textit{am}) exp \left ( - \frac{\lambda \Delta t}{A \rho c} \left (\gamma_i +\frac{1}{2}+\frac{S_i-R_i}{\dot{m}_{s,i} \Delta t}\right) \right)$},
\label{eq:outlet_temp}
\end{equation}
where $\lambda$ denotes heat transfer coefficient of the pipe and $\tau_i^\textit{am}$ is the ambient temperature. 

The parameters of the DHS are listed in Table~\ref{table:dhs_parameters}. The CHP points $(H_i, P_i)$ for $i \in \{1,2,3,4\}$ are taken from~\cite{abdollahi2014optimization}. The hourly heat demand data is obtained from the dataset provided in~\cite{ruhnau2019time}. The original heat demand data was scaled such that the maximum heat demand is 67 MW, while the maximum heat production of the CHP is 70 MW. This was done to sufficiently stretch the heat production, while avoiding violations resulting from close heat demand and heat production values.
\begin{table}[t!]
\centering
\caption{The district heating system parameters.}
\vspace*{7mm}
\begin{tabular}{lr}
\toprule
Parameter & Value\\
\midrule
$(H_i,P_i)[MW]$ & $(0,10),(10,5),(70,35),(0,50)$\\
$T [h]$ &$12$\\
$a_0[e/h]$ & $8.1817$ \\
$a_1[e/h]$& $38.1805$ \\
$\Delta t [h]$ & $1$\\
$L [km]$ &$4$ \\
$A [m^2]$ & $1.1$ \\
$c [J/(kg ^{\circ} C)]$ &$4181.3$\\
$\lambda [W/(m ^{\circ} C)]$ & $0.735$\\
$\rho$ & $963$ \\
\bottomrule
\end{tabular}
\label{table:dhs_parameters}
\end{table}
\section*{Appendix B}
This section provides details on constructing a penalty function for a simplified model of the DHS. To construct a delivered heat function that is non-decreasing in the produced heat decision variables and facilitates computationally stable optimization, we assume a constant and zero ambient temperature and a zero temperature in the return pipe. Under these assumptions, by expressing the mass flow from equation (\ref{eq:produced_heat}), and substituting the outlet temperature without heat loss as defined in equation (\ref{eq:no_temp_loss}) into equation (\ref{eq:outlet_temp}), we derive the following expression for the delivered heat function:

\begin{align}
\begin{split}
    &y_i =  c \Bigg ( \left(R_i-\rho A L \right) \tau_{s,i-\gamma_i}^{\textit{in}}+\sum_{k=i-n_{w,i}+1}^{i-\gamma_i-1} \frac{h_k}{c}+\nonumber\\& +\left ( \frac{h_i}{c \tau_i^\textit{in}} +\rho A L - S_i \right ) \tau_{i-n_{w,i}}^\textit{in} \Bigg )\\&exp \left ( - \frac{\lambda \Delta t}{A \rho c} \left (\gamma_i +\frac{1}{2}+\frac{S_i-R_i}{\dot{m}_{s,i} \Delta t}\right) \right)
\end{split}
\end{align}
Finally, we define the penalty function $\hat{J}$ as:
\begin{equation}
    \hat{J} = J+ C \sum_{i=1}^T\left( y_i-q_i\right)^2
\end{equation}

\section*{Appendix C}
In this section, we detail constructing a penalty function for a model of the DHS derived from learned deep neural networks. The DHS with parameters defined in Appendix A is also used in this application domain.
The process of building an optimization model, illustrated in Figure~\ref{fig:empirical_learning}, consists of five steps.

\begin{figure}[t!]
\centering
\includegraphics[width=1\linewidth]{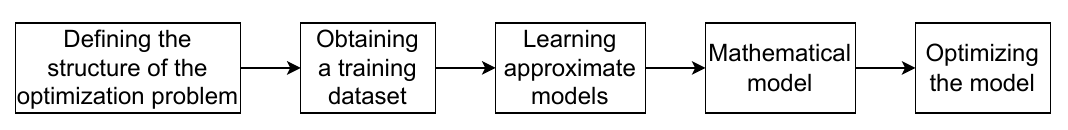}
\caption{A model derived from deep neural networks.}
\vspace*{7mm}
\label{fig:empirical_learning}
\end{figure}

In the first step, we identify processes to approximate using deep neural networks, and inputs and outputs of these networks. We approximate the state transition function $g$ modeling the transition from the previous state and previous decision variables to the current state, $s_i = g(s_{i-1}, h_{i-n_w},p_{i-n_w,}\ldots,h_i,p_i)$, $i=1,\ldots,T$, and the system output function $f$ modeling the transition from the current state and previous decision variables to the system output, delivered heat to the consumer, $y_i = f(s_i,h_{i-n_w},p_{i-n_w},\ldots, h_i,p_i)$. Therefore, the delivered heat to the consumer is the composition of functions $g$ and $f$. The state of the system $s_i$ is partially observable, and includes observations of the temperature at the inlet and outlet of the supply pipe and the mass flow, $s_i = \left(\tau^{\textit{in}}_{s,i}, \tau^{\textit{out}}_{s,i}, \dot{m}_{s,i}\right)$. The parameter $n_w$, representing the number of previously produced heats that influence the current state and delivered heat, is empirically determined to be eleven for a pipe length of $4km$.

In the second step, we generate the data for training deep neural networks modeling the state transition $g$ and system output $f$ functions. For the heat demand, we use hourly dataset from~\cite{ruhnau2019time}, covering a five-year period from 2015 to 2019 during the heating season (mid-November to mid-March). Firstly, the data for warm-up training is generated using a fast linear optimizer~\cite{abdollahi2014optimization,gu2017modeling} with the GridPenguin simulator~\cite{wu2022gridpenguin}. This dataset has 1000 rows. Then, the training data is obtained using mixed integer linear optimizer~\cite{giraud2017optimal} with the same simulator. This dataset is then divided, allocating 80\% for training (10535 rows) and 20\% for testing (2634 rows).

The third step involves learning deep neural networks using the obtained dataset. For this, we first normalize the input and output data of the neural networks to a range between zero and one. To model non-decreasing constraint functions, these neural networks should be non-decreasing in all decision variables. This is achieved by constraining their input weights to be non-negative in all decision variables, all weights between layers to be non-negative and using non-decreasing activation functions, such as rectified linear units.

The size of these neural networks is particularly important since it affects the modeling capability and the computational time of the optimization algorithm. To determine suitable size, we train nine neural networks differing in the number of between layers and the number of neurons per between layers. The configurations of these neural networks are $[1]$,$[1,1]$,$[3]$,$[5]$,$[5,3]$,$[10]$,$[10,10]$,$[50,50]$,$[100,100,100]$. In these configurations, the first number indicates the number of neurons in the first between layer, the second number represents the number of neurons in the second between layer, and so on. In Figure~\ref{fig:predictions}, we plot the root mean squared error for one-step predictions on a testing dataset, for temperatures at the pipe's inlet and outlet, the mass flow, and the delivered heat to the consumer. The ranges of these variables are: inlet temperature $\tau_s^\textit{in} \in[70^{\circ} C,120^{\circ} C]$, outlet temperature $\tau_s^\textit{out} \in[70^{\circ} C,115^{\circ} C]$, mass flow $\dot{m}_s \in [5kg/s, 810kg/s]$ and delivered heat $y \in[5 MW,65MW]$.

\begin{figure*}[t!]
\centering
\subfigure[Temperature at the inlet of the supply pipe.]{
  \includegraphics[width=0.42\linewidth]{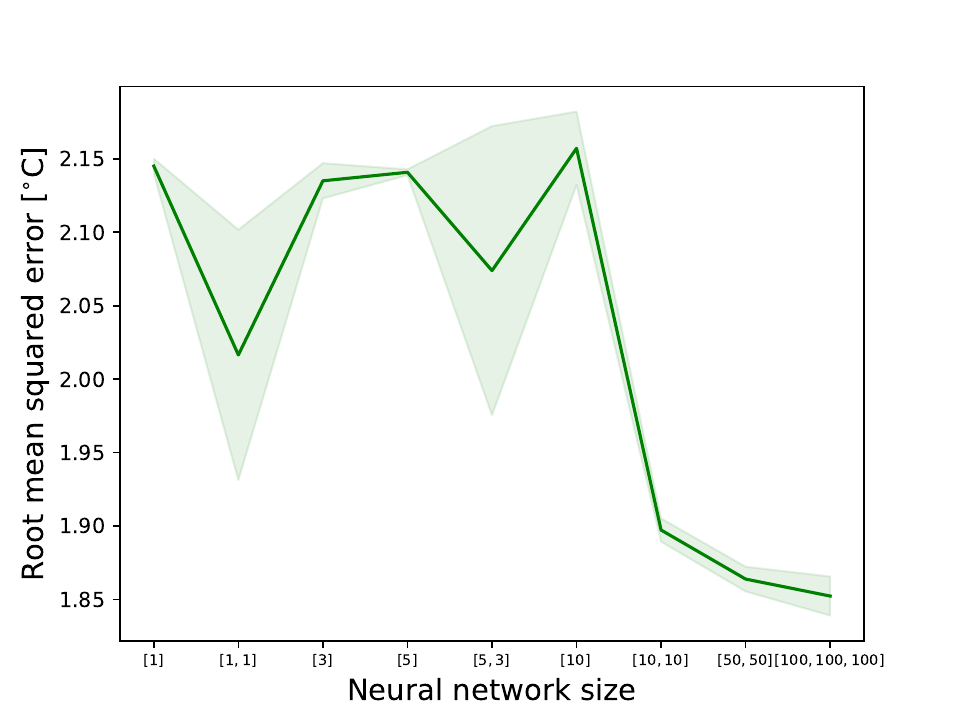 }}
\subfigure[Temperature at the outlet of the supply pipe.]{
  \includegraphics[width=0.42\linewidth]{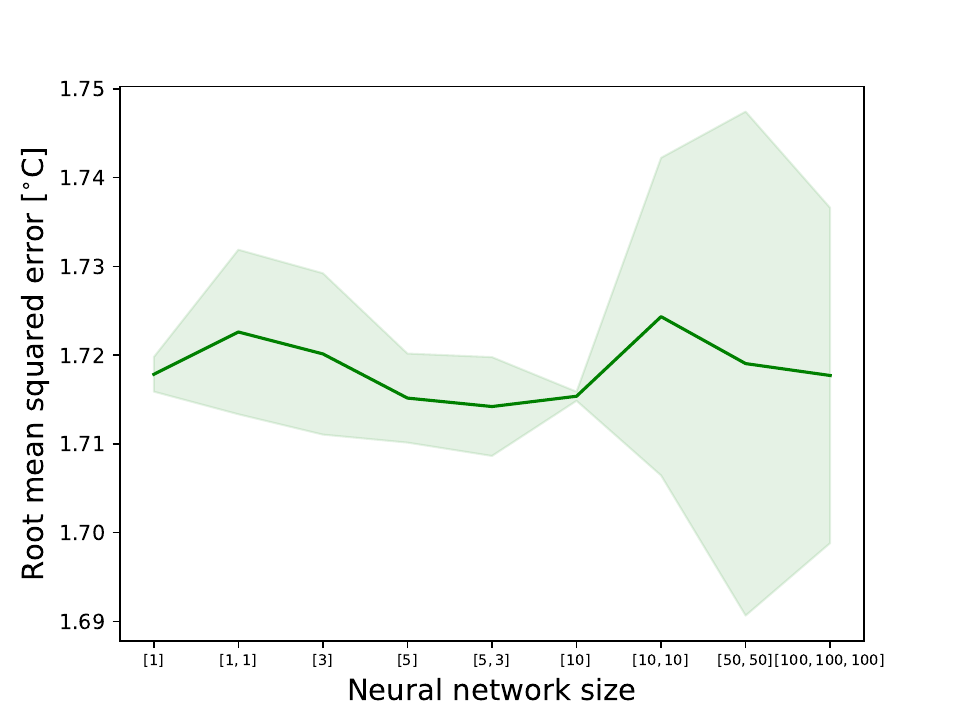}}

\subfigure[Mass flow.]{
  \includegraphics[width=0.42\linewidth]{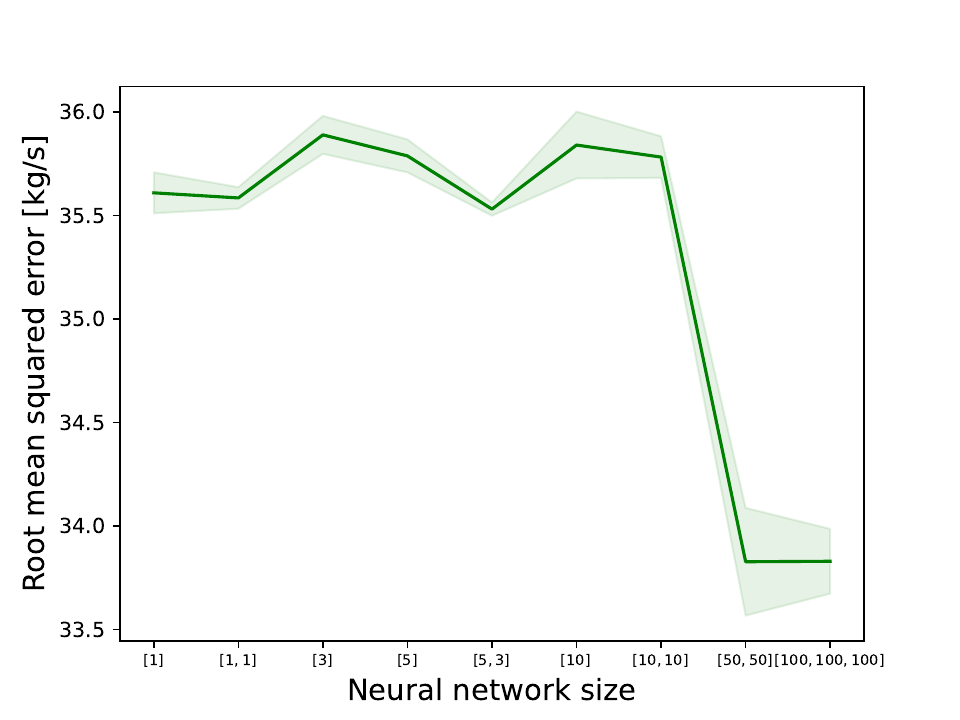}}
\subfigure[Delivered heat.]{
  \includegraphics[width=0.42\linewidth]{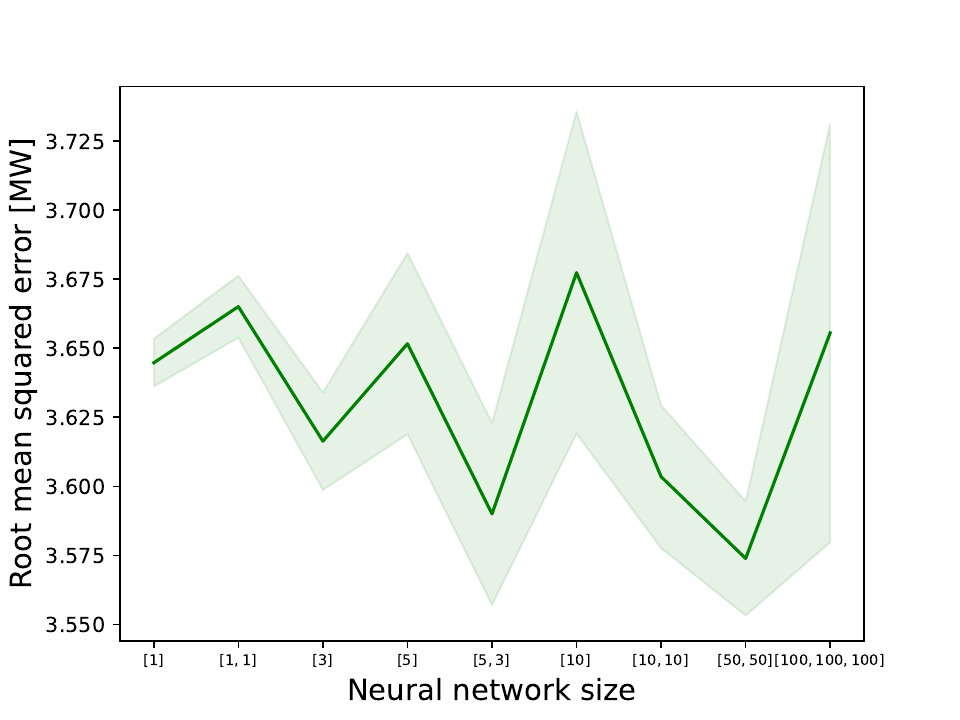} 
}
\caption{Root mean squared errors (y-axis) of one-step predictions with standard deviation over five repetitions for different neural network sizes (x-axis).}
\vspace*{7mm}
\label{fig:predictions}
\end{figure*}
Given its low prediction error for inlet temperature, mass flow, and delivered heat variables, we selected the neural network of size 
$[50,50]$ to approximate both the state transition function $g$ and the system output function $f$. For training these deep neural networks, we use the Adam optimizer ~\cite{kingma2017adam} with a learning rate of 0.001. Other hyperparameters, such as the maximum number of training epochs $M$, and the early stopping parameters -- delta $\hat{\Delta}$ and patience $\hat{M}$ proposed in~\cite{bengio2012practical}, are empirically determined. For the state transition function, their values are $M=3000$, $\hat{\Delta}=10^{-6}$, $\hat{M}=200$. For the system output function, their values are $M=1000$, $\hat{\Delta}=5\cdot 10^{-6}$, $\hat{M}=35$.

In the fourth step, we chain these fixed learned deep neural network models together for a set planning horizon $T$. For the optimization using the PGA, these learned models are compiled in a recurrent neural network unrolled for the fixed planning horizon. As the delivered heat to the consumer $y_i$ is the composition of functions $f$ and $g$, it is defined as:
\begin{align}
\begin{split}
    &y_i = f(g(g(g(..g()), h_{i-n_w-1},p_{i-n_w-1},\ldots, h_{i-1}, p_{i-1}), \\& h_{i-n_w},p_{i-n_w},\ldots, h_i, p_i), h_{i-n_w},p_{i-n_w},\ldots, h_i, p_i),
    \end{split}
\end{align}
where the number of functions $g$ in the composition is equal to time step $i$.

Similarly to the previous application domain, the penalty function is formulated as follows:
\begin{equation}
    \hat{J} = J + C \sum_{i=1}^T (y_i - q_i)^2,
\end{equation}
where the objective function $J$ is defined in equation (\ref{eq:objective_function}).

\section*{Appendix D}
An N-dimensional artificial application domain has three decision variables $x$,$y$, $z$ and it is characterized by the following objective function:
\begin{equation}
    J = x+y+z
\end{equation}
And the following constraints:
\begin{equation}
    \begin{split}
        e^{0.1+0.75x} \geq 15\\
        e^{0.05+x+0.5y} \geq 100\\
        e^{0.1x+0.5y+z} \geq 10
    \end{split}
\end{equation}
\section*{Appendix E}
In this section, we present details on the experiments concerning solving penalty functions for an N-dimensional problem, a simplified model of the DHS, and a model derived from deep neural networks. 

To assess the computational tractability, we randomly initiate the optimization with twenty feasible initial solutions for each domain. These initial solutions are listed in Table~\ref{tab:initial_solutions}.

\begin{table*}[t!]
\caption{Feasible initial solutions per application domain.}
\vspace*{7mm}
\begin{tabular}{rrr}
\toprule
 N-dimensional problem&District heating system [MW]&Deep neural networks [MW]\\
\midrule
$(4,2,2)$&$(66, 68, 60, 65, 64, 60, 70, 65, 62, 64, 66, 70)$&$(49, 67, 50, 68, 68, 43, 66, 55, 56, 47, 40, 32)$\\
$(4,3,2)$&$(69, 63, 61, 63, 70, 65, 68, 62, 62, 68, 61, 64)$&$(67, 34, 52, 49, 48, 31, 34, 34, 54, 62, 65, 59)$\\
$(4,4,2)$&$(62, 69, 64, 63, 70, 60, 62, 64, 60, 65, 63, 67)$&$(32, 40, 50, 58, 64, 36, 56, 65, 45, 32, 68, 34)$\\
$(4,5,2)$&$(70, 68, 70, 65, 61, 66, 69, 60, 60, 66, 65, 69)$&$(47, 42, 38, 65, 48, 42, 47, 31, 38, 34, 35, 37)$\\
$(4,5,3)$&$(64, 66, 67, 62, 64, 65, 68, 64, 64, 66, 64, 63)$&$(66, 34, 56, 55, 52, 42, 55, 61, 53, 35, 62, 61)$\\
$(4,3,4)$&$(63, 64, 70, 62, 70, 64, 61, 61, 68, 62, 69, 63)$&$(68, 51, 45, 52, 48, 59, 65, 36, 69, 49, 60, 50)$\\
$(4,3,5)$&$(60, 67, 66, 68, 65, 62, 60, 62, 65, 66, 63, 64)$&$(33, 57, 69, 52, 64, 59, 56, 59, 48, 38, 43, 34)$\\
$(4,4,4)$ &$(67, 67, 63, 68, 68, 70, 68, 69, 65, 61, 65, 60)$&$(31, 44, 66, 68, 44, 53, 31, 56, 61, 38, 43, 50)$\\
$(4,5,4)$ &$(67, 60, 68, 65, 69, 66, 62, 70, 66, 65, 62, 62)$&$(68, 42, 31, 35, 59, 52, 62, 33, 31, 70, 38, 54)$\\
$(4,4,5)$ &$(70, 70, 69, 65, 61, 65, 66, 64, 68, 64, 66, 70)$&$(67, 62, 60, 62, 63, 40, 69, 45, 31, 54, 40, 59)$ \\
$(4,5,5)$ &$(62, 64, 63, 67, 68, 63, 69, 62, 65, 67, 65, 64)$& $(45, 50, 58, 44, 44, 40, 37, 70, 43, 68, 35, 35)$\\
$(5,3,4)$&$(64, 60, 64, 60, 70, 66, 69, 70, 65, 70, 70, 67)$& $(69, 57, 43, 61, 39, 31, 49, 30, 49, 48, 30, 31)$\\
$(5,4,4)$ &$(64, 65, 65, 69, 67, 69, 68, 69, 62, 69, 66, 61)$& $(66, 41, 45, 63, 64, 50, 53, 46, 40, 58, 56, 70)$\\
$(5,5,4)$ &$(67, 64, 64, 70, 65, 60, 63, 70, 65, 66, 61, 66)$&$(52, 70, 44, 50, 64, 66, 31, 53, 50, 44, 34, 54)$ \\
$(5,4,5)$ &$(67, 67, 60, 69, 68, 68, 68, 60, 70, 69, 62, 70)$& $(36, 38, 53, 38, 33, 44, 65, 32, 35, 60, 46, 30)$\\
$(5,5,5)$ &$(66, 69, 68, 60, 66, 70, 65, 61, 60, 60, 62, 69)$&$(52, 31, 67, 30, 31, 38, 47, 70, 50, 70, 30, 49)$\\
$(5,6,6)$ &$(70, 70, 68, 63, 68, 63, 60, 61, 61, 66, 70, 62)$&$(66, 59, 60, 33, 40, 42, 62, 32, 51, 53, 57, 37)$\\
$(6,5,5)$ &$(64, 69, 61, 66, 61, 68, 60, 69, 60, 61, 69, 66)$&$(49, 33, 63, 43, 32, 49, 69, 51, 48, 45, 53, 58)$\\
$(5,6,5)$ &$(66, 70, 62, 65, 61, 66, 69, 68, 61, 68, 62, 62)$&$(69, 48, 62, 45, 50, 66, 68, 66, 46, 69, 58, 55)$\\
$(6,5,6)$&$(62, 69, 65, 64, 60, 65, 60, 62, 65, 61, 65, 66)$&$(53, 62, 69, 35, 38, 54, 69, 51, 62, 38, 42, 33)$\\
\bottomrule
\end{tabular}%
\label{tab:initial_solutions}
\end{table*}
A simplified model of the DHS optimizes twelve hours of a typical day during winter season, while the model derived from learned deep neural networks optimizes twelve hours of a day during early spring season. Heat demands of these days are plotted in Figure~\ref{fig:heat_demands}. Their values influence both the operational cost and feasible initial solutions. In Figures 3 and 4 in the main paper, it can be observed that the operational cost acquired during the winter is significantly higher. In Table~\ref{tab:initial_solutions}, it can be observed that the range of values for initial solutions is smaller for the winter day, due to higher heat demands. 
\begin{figure*}[t!]
\centering
\subfigure[Heat demand during twelve hours of winter season.]{
  \includegraphics[width=0.42\linewidth]{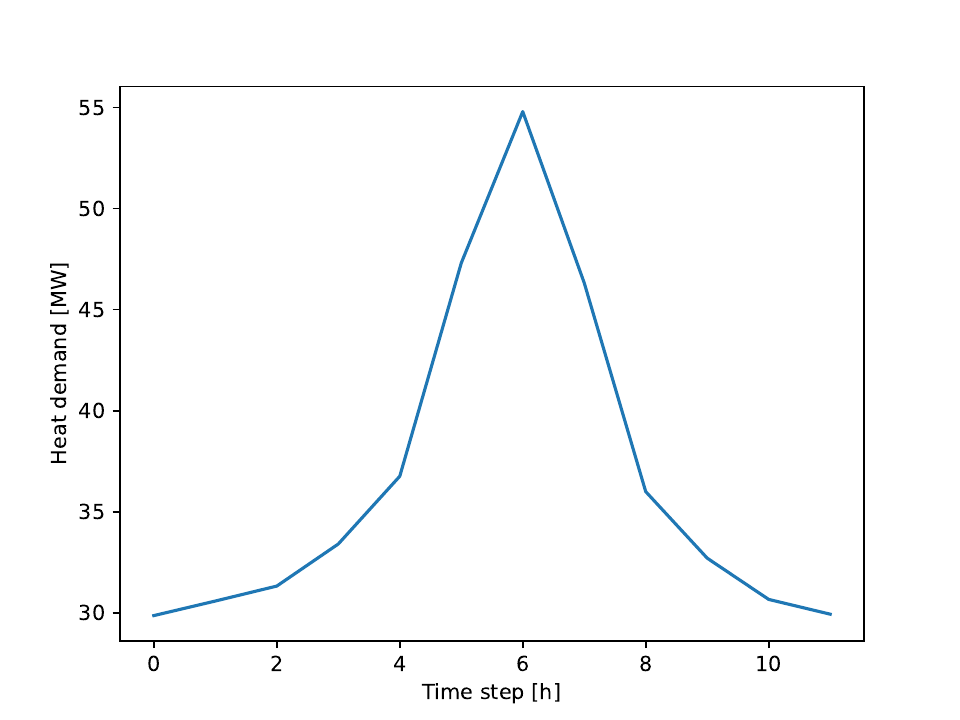}}
\subfigure[Heat demand during twelve hours of early spring season.]{
  \includegraphics[width=0.42\linewidth]{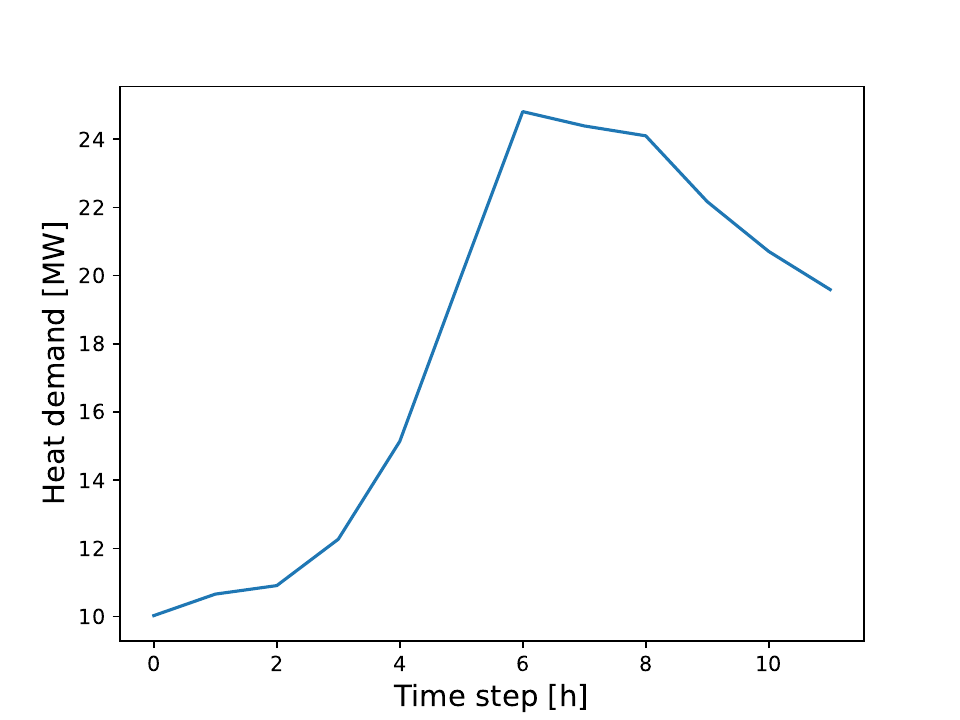}}
\caption{Heat demands.}
\vspace*{7mm}
\label{fig:heat_demands}
\end{figure*}

Finding the right strength for parameter C is challenging. For a large strength, the optimization converges to a poor local minimum, a solution that resides on or near constraints but with a sub-optimal value of the objective function. As constraint values are near zero, the update to the guardrail variable has no effect. Therefore, the optimization will remain in a poor local minimum. For a small value of parameter C, the optimization converges to a deep local minimum, a solution that violates constraints with an optimistically low value of the objective function. Using the guardrail variable, we recover from these constraint violations. However, if parameter C is too small, addressing violations might require high computational time. In Tables~\ref{table:param_C_N},~\ref{table:param_C_simplified_DHS} and~\ref{table:param_C_dnn}, we evaluate the objective function values and the maximum absolute value of violated constraints,$\gamma_\textit{max}$, in the solutions of the first outer iteration. This evaluation is performed for a small sample of randomly selected strengths of parameter C, when solving an N-dimensional optimization problem, a simplified model of a DHS and a model derived from deep neural networks, respectively.

\begin{table}[t!]
\centering
\caption{Objective function and constraint values for parameter C when solving an N-dimensional optimization problem.}
\vspace*{7mm}
\begin{tabular}{lrr}
\toprule
C & J & $\gamma_\textit{max}$\\
\midrule
$0.0005$&$-1.47$&$-21$ \\
$0.05$&$6.39$& $-1.12$\\
$5$&$6.51$&$-0.01$ \\
\bottomrule
\end{tabular}
\label{table:param_C_N}
\end{table}

\begin{table}[t!]
\centering
\caption{Objective function and constraint values for parameter C when solving a simplified model of a district heating system.}
\vspace*{7mm}
\begin{tabular}{lrr}
\toprule
C & J [\euro] & $\gamma_\textit{max}$ [MW]\\
\midrule
$1$&$2973$&$-32.71$ \\
$100$&$9231$& $-6.24$\\
$10000$&$10326$&$-0.08$ \\
\bottomrule
\end{tabular}
\label{table:param_C_dnn}
\end{table}

\begin{table}[t!]
\centering
\caption{Objective function and constraint values for parameter C when solving an optimization model derived from learned deep neural networks.}
\vspace*{7mm}
\begin{tabular}{lrr}
\toprule
C & J [\euro] & $\gamma_\textit{max}$ [MW]\\
\midrule
$1$&$485$&$-26.85$ \\
$100$&$1840$& $-4.91$\\
$10000$&$1984$&$-0.06$ \\
\bottomrule
\end{tabular}
\label{table:param_C_simplified_DHS}
\end{table}
Selected strengths for parameter C are 0.05, 100 and 100 for an N-dimensional problem, a simplified model of a district heating system and a model derived from deep neural networks, respectively, as they do not result in large constraint violations nor in satisfied constraints in the first outer iteration.
\end{document}